\documentclass[12pt, reqno]{amsart}
\usepackage{amsfonts}
\usepackage{bbm}
\usepackage{amscd}
\usepackage{amssymb, eucal, amsmath, xypic, latexsym}
\usepackage{pifont}
\usepackage{color}
\usepackage{amsthm,indentfirst,bm,fancyhdr,dsfont}
\usepackage{graphicx}
\usepackage[all]{xy}
\usepackage[CJKbookmarks=true]{hyperref}

\usepackage{mathrsfs}
\usepackage{hyperref}

\newtheorem{theorem}{Theorem}[section]
\newtheorem{lemma}[theorem]{Lemma}
\newtheorem{prop}[theorem]{Proposition}
\newtheorem{proposition}[theorem]{Proposition}
\newtheorem{corollary}[theorem]{Corollary}
\theoremstyle{definition}

\newtheorem{defn}[theorem]{Definition}
\newtheorem{definition}[theorem]{Definition}

\newtheorem{remark}[theorem]{Remark}

\newtheorem{question}[theorem]{Question}

\numberwithin{equation}{section}

\def\ggg{\mathfrak{g}}

\def\gl{\mathfrak{gl}}
\def\sll{\mathfrak{sl}}

\def\calh{\mathcal{H}}

\def\cale{\mathcal{E}}

\def\ggg{\mathfrak{g}}

\def\fsl{\mathfrak{S}_l}
\def\fsm{\mathfrak{S}_m}

\def\bbq{\mathbb{Q}}

\def\buq{{\textbf{U}}_q}

\def\scrW{\mathscr{W}}

\def\End{{\mathrm{End}}}

\def\GL{\text{GL}}

\def\id{\mathsf{id}}

\def\uV{\underline{V}}
\def\uVm{\underline{V}^{\otimes m}}

\def\tsgm{\textsf{G}_{\text{m}}}
\newcommand{\dha}{\mathfrak{H}\kern -.6em\mathfrak{H}}

\def\tsb{\mathrm{\mathsf{B}}}
\def\tsa{\textsf{A}}

{\vskip-\lastskip\medskip
	\noindent
	{\em #1.}\enspace
}%
{\qed\par\medskip
}
\begin{document}
	\title[duplex Hecke Algebras of type $B$]
	{duplex Hecke Algebras of type $B$}
	\author{Yu Xie, An Zhang and Bin Shu}
	\address{School of Mathematical Sciences, East China Normal University, No. 500 Dongchuan Rd., Shanghai 200241, China}
	\email{xieyuyoux@163.com}
	\address{School of Mathematical Sciences, East China Normal University, No. 500 Dongchuan Rd., Shanghai 200241, China}
    \email{52215500003@stu.ecnu.edu.cn}
    \address{School of Mathematical Sciences, Ministry of Education Key Laboratory of Mathematics and Engineering Applications \& Shanghai Key Laboratory of PMMP,  East China Normal University, No. 500 Dongchuan Rd., Shanghai 200241, China}
    \email{bshu@math.ecnu.edu.cn}
	
	\thanks{This work is supported by the National Natural Science Foundation of China (Grant No. 12071136 and 12271345), and by Science and Technology Commission of Shanghai Municipality (No. 22DZ2229014).}

	\begin{abstract} As a sequel to \cite{XZ}, in this article we first introduce a so-called duplex Hecke algebras of type $\tsb$ which is  a $\bbq(q)$-algebra associated with the Weyl group $\mathscr{W}(\tsb)$ of type $\tsb$, and symmetric groups $\mathfrak{S}_l$ for $l=0,1,\ldots,m$, satisfying some Hecke relations (see Definition \ref{defn: duplex H}).
		This notion originates from the degenerate duplex Hecke algebra arising from the course of study of a kind of Schur-Weyl duality of Levi-type (see \cite{SXY}), extending the duplex Hecke algebra
		of type $\tsa$  arising from the related $q$-Schur-Weyl duality of Levi- type (see \cite{XZ}).
		%
		A duplex Hecke algebra of type
		$\tsb$ admits natural representations on certain tensor spaces. We then establish a Levi-type $q$-Schur-Weyl duality of type $\tsb$, which reveals the double centralizer property between such duplex Hecke algebras and $\imath$quantum groups studied by Bao-Wang in \cite{BW}.
	\end{abstract}
	\maketitle

	\section*{Introduction}

	\subsection{Background} An algebraic group $G$ is called a semi-reductive group if $G$ is
	a semi-direct product of a reductive closed subgroup $G_{0}$ and
	the unipotent radical $U$. It is important to study semi-reductive
	algebraic groups and their Lie algebras in lots of cases when the
	underground filed is of characteristic $p>0$ (see \cite{OSY}). Set $\underline{G}=G\times_{\nu}V$,
	where $G=\GL(V)$ and $\nu$ be the natural representation of $G$
	on $V$. Then this enhanced group $\underline{G}$ naturally becomes a semi-reductive group (called an enhanced reductive group by Shu-Xue-Yao \cite{SXY}). In order to study
	polynomial representations of general linear groups, classical Schur
	algebras are produced. By analogy of this, the $m$th tensor representations
	of an enhanced group $\underline{G}$ naturally produce the so-called enhanced
	Schur algebra $\cale(n,m)$. In the course of study
	in representations of enhanced Schur algebras,  an algebraic model of degenerate duplex Hecke algebras is introduced  in \cite{SXY}, where it is
	denoted by $\mathcal{H}_{m}$. Roughly speaking, for $\underline{V}$ a one-dimensional extension of $V$ associated with $\GL(V)$, Shu-Xue-Yao established the following duality in \cite{SXY}:	
		$$\Xi(\mathcal{H}_{m})=\End_{\mathbb{C}\Phi(\GL_{n}\times
			\tsgm)}(\underline{V}^{\otimes m});$$
		$$\End_{\Xi(\mathcal{H}_{m})}(\underline{V}^{\otimes m})=\mathbb{C}\Phi(\GL_{n}\times \tsgm),$$	
	where $\GL_n$ and $\tsgm$ are the general linear group and the one-dimensional multiplication group, respectively, and  $\GL_n\times \tsgm$ is naturally regarded as a connected closed subgroup of $\underline{G}$,
	and $\Phi$ and $\Xi$ denote the natural representation of $\underline{G}$ and $\mathcal{H}_m$ on $\underline{V}^{\otimes m}$ respectively.	
	This duality is called a Levi-type Schur-Weyl duality.

	Recall that soon after quantum groups arising, Jimbo \cite{Jim86}
    established a duality between quantum groups and Hecke algebras of type $\tsa$,  which	is a $q$-deformation of classical Schur-Weyl duality. This duality is called $q$-Schur duality or Jimbo-Schur duality nowadays.
	
	Combining Jimbo-Schur duality and  Levi-type Schur-Weyl duality mentioned above, Xue-Zhang \cite{XZ} introduced duplex Hecke algebras (of type $\tsa$) and established the $q$-deformation of Levi-type Schur-Weyl duality\footnote{In \cite{XZ}, it is called the doubled Hecke algebra. Here that notion is modified into the current one ``duplex Hecke algebras," the latter of which can reflect more what it looks like. More  explanation about this notion can be refereed to \cite{SXY}.}. There is a natural question which seems nontrivial:
	\begin{question} \label{ques: 1.1}
		How to extend Xue-Zhang's work to the case of type $\tsb$, i.e.
		how to define doubled Hecke algebras of type $\tsb$, and how to establish a $q$-deformation of Levi-type Schur-Weyl duality?
	\end{question}

	\subsection{Our purpose}
	In this paper, our purpose is to answer the above question. We formulate a so-called duplex Hecke algebra of type $\tsb$, and establish a related $q$-deformation of Levi Schur-Weyl duality of type $\tsb$  by exploiting $\imath$quantum groups.

	Recall that a symmetric pair $(\mathfrak{g},\mathfrak{g}^{\theta})$ consists
	of a semisimple Lie algebra $\ggg$ and its fixed-point subalgebra under $\theta$ which is an involution on $\ggg$.
	The classification of irreducible symmetric pairs is equivalent to
	the classification of real forms of complex simple Lie algebras. Let
	$\textbf{U}=\textbf{U}_q(\mathfrak{g})$ be a quantum group associated with $\ggg$. According
	to Gail Letzter (cf. \cite{Letz}), a quantum symmetric pair $(\textbf{U},\textbf{U}_{q}^{\imath})$
	consists of a quantum group \textbf{U} and its right coideal subalgebra
	$\textbf{U}_{q}^{\imath}$ which specializes at $q\rightarrow1$ to	$\mathbf{U}(\mathfrak{g})$ and $\mathbf{U}(\mathfrak{g^\theta})$.
	We call $\textbf{U}_{q}^{\imath}$ an $\imath$quantum
	group. Bao-Wang (and Bao-Wang-Watanabe) establish a kind of
	Schur duality (called $\imath$Schur duality) between an $\imath$quantum
	group of type $\tsa$ and a Hecke algebra of type $\tsb$ (see \cite{BW, BWW}).

	\subsection{Main results}
	\subsubsection{} We establish a duplex Hecke algebra $\dha_m$ of type $\tsb$, which is a $\bbq(q)$-algebra associated with the Weyl group $\mathscr{W}(\tsb)$ of type $\tsb$ and symmetric groups $\mathfrak{S}_l$ for $l=0,1,\ldots,m$, satisfying some Hecke relations (see Definition \ref{defn: duplex H}).
	
	\subsubsection{}
	Based on  Bao-Wang's work along theirs jointly with Watanabe	 in  \cite{BW, BWW},
	we  introduce  a Levi-type subalgebra ${L}_{q}^{\imath}(\sll_{2r+4})$ of $\imath$quantum group $\textbf{U}_{q}^{\imath}(\sll_{2r+4})$ (see \S\ref{Levi-type}). This  becomes another important ingredient in our argument solving Question \ref{ques: 1.1}.
	
	\subsubsection{} Let $V$ be the $\bbq(q)$-vector space of dimension $2r+2$, and	let $\uV$ be a $2$-dimensional extension of $V$.
	Then there are both natural actions $\Phi$ and $\Xi$ of $L_{q}^{\imath}(\sll_{2r+4})$
	and $\dha_m$ on $\uV^{\otimes m}$ respectively (see \S\ref{natural rep} and \S\ref{sec action of h}).
	We  prove the following double centralizer theorem (see Theorems \ref{thm: main thm1} and  \ref{thm: main thm2}).
	\begin{align*}
		\Phi(L_{q}^{\imath}(\mathfrak{s}\mathfrak{l}_{2r+4}))&=
		\End_{\dha_m}(\underline{V}^{\otimes m});\cr
		\Xi(\dha_m)&=\End_{L_{q}^{\imath}
			(\mathfrak{s}\mathfrak{l}_{2r+4})}(\underline{V}^{\otimes m})^{\textsf{op}}.
	\end{align*}
	
	\subsection{The structure of the paper}
	The article contains three sections. In the first section,  we recall  some  basic concepts, including quantum groups, $\imath$quantum group,
	Hecke algebras and $\imath$Schur duality.
	In the second section, we introduce the duplex Hecke  algebra  $\dha_m$  of type $\tsb$ and define its representation on $\underline{V}^{\otimes m}$, and show that  this representation is well-defined. In the third section, we define a Levi-type subalgebra $L_{q}^{\imath}(\sll_{2r+4})$ of $\imath$quantum group $\textbf{U}_{q}^{\imath}(\mathfrak{s}\mathfrak{l}_{2r+4})$. We then establish the $\imath$Schur duality of Levi-type.

   Throughout the paper,  $q$ is a fixed  parameter which is assumed to be transcendental over $\bbq$.
   It is worth mentioning that the proof of Theorem \ref{thm: main thm1} and Lemma \ref{lem: 3.3} are dependent on  the condition of $q$ being transcendental over $\bbq$ (see \cite{DJ}).

\section{Preliminaries}

In this section, we recall some basic notions and introduce notations used in this article, particularly involving $\imath$quantum groups. For details, one can be referred to \cite{BW}, \cite{BWW} or \cite{LW}.
\subsection{Quantum groups }

\begin{defn} The quantum group $\textbf{U}_{q}(\mathfrak{g}\mathfrak{l}_{n})$ associated with $\gl_n$
	is the associative algebra generated by	all $E_{i}$, $F_{i}$, $D_{k}$, $D_{k}^{-1}$ with
    $1\leqslant i\leqslant n-1$ and $1\leqslant k\leqslant n$
	over ${\mathbb{Q}}(q)$, which satisfies the following relations
	
	\begin{align*}
		&D_{k}D_{k}^{-1}=1=D_{k}^{-1}D_{k};\;D_{k}D_{s}=D_{s}D_{k} \;(1\leq k,s\leq n);
		\cr
		&D_{i}E_{i}D_{i}^{-1}=qE_{i};\;D_{i}F_{i}D_{i}^{-1}=q^{-1}F_{i};\cr
		&D_{i+1}E_{i}D_{i+1}^{-1}=q^{-1}E_{i};\;D_{i+1}F_{i}D_{i+1}^{-1}=qF_{i};\cr
		&D_{k}E_{i}D_{k}^{-1}=E_{i}\;(k\neq i,i+1);\;D_{k}F_{i}D_{k}^{-1}=F_{i}\;(k\neq i,i+1);\cr
		&E_{i}F_{j}-F_{i}E_{j}=\delta_{ij}
		\frac{D_{i}D_{i+1}^{-1}-D_{i}^{-1}D_{i+1}}{q-q^{-1}};\cr
		&
		\begin{cases}
			E_{i}E_{j}=E_{j}E_{i}\cr
			F_{i}F_{j}=F_{j}F_{i}
		\end{cases}
		\text{ if } c_{ij}=0;
	\end{align*}
	and
	\begin{align*}
		\begin{cases}
			E_{i}^{2}E_{j}-(q+q^{-1})E_{i}E_{j}E_{i}+E_{j}E_{i}^{2}=0\cr
			F_{i}^{2}F_{j}-(q+q^{-1})F_{i}F_{j}F_{i}+F_{j}F_{i}^{2}=0
		\end{cases}
		\text{ if } c_{ij}=-1,
	\end{align*}
	where $C=(c_{ij})_{(n-1)\times(n-1)}$ is the  Cartan matrix of $\mathfrak{s}\mathfrak{l}{}_{n}$.
\end{defn}

\begin{defn} The quantum group
	$\buq(\sll_{n})$ associated with $\sll_n$
	is the $\mathbb{Q}(q)$-subalgebra of $\textbf{U}_{q}(\mathfrak{g}\mathfrak{l}_{n})$
	generated by	
	all $E_{i},F_{i},K_{i}:=D_{i}D_{i+1}^{-1}$ with $1\leq i\leq n-1.$
\end{defn}
\begin{remark}
	The quantum group $\buq(\gl_n)$
	is a Hopf algebra with comultiplication $\Delta:\buq
	(\gl_n)\rightarrow \buq(\gl_n)\otimes \buq(\gl_n)$
	given by
	\begin{align*}
		\Delta(D_{i})=D_{i}\otimes D_{i},\Delta(E_{i})=1\otimes E_{i}+E_{i}\otimes K_{i}^{-1},\Delta(F_{i})=F_{i}\otimes1+K_{i}\otimes F_{i}.
	\end{align*}
\end{remark}

\subsection{$\imath$Quantum groups}
Let $r$ be a given positive integer. Set
\begin{align*}
	\mathbb{I}_{2r+2}&:=\{-r-\frac{1}{2},-r+\frac{1}{2},\ldots, -\frac{1}{2},\frac{1}{2},\ldots, r-\frac{1}{2},r+\frac{1}{2}\},\cr
	\mathbb{I}_{2r+1}&:=\{-r,-r+1,\ldots, 0,\ldots, r-1,r\}.
\end{align*}

\begin{defn} Index the generators
	(or the simple roots) for $\buq(\mathfrak{s}\mathfrak{l}_{2r+2})$
	by $\mathbb{I}_{2r+1}$. Then the $\imath$quantum group $\buq^{\imath}(\sll_{2r+2})$  is  defined as the $\bbq(q)$-subalgebra of
	$\buq(\gl_{2r+2})$ generated by all
	\begin{alignat*}{2}
		k_{i}&:= K_{i}K_{-i}^{-1}\;&&(i\in\mathbb{I}_{2r+1}),\\
		B_{i}&:= E_{i}+F_{-i}K_{i}^{-1}\;&&(0\neq i\in\mathbb{I}_{2r+1}),\\
		B_{0}&:= E_{0}+qF_{0}K_{0}^{-1}+K_{0}^{-1}.
	\end{alignat*}
\end{defn}

\subsubsection{Natural modules}\label{natural rep} Let $V$ be the $\bbq(q)$-vector space spanned by $\{\eta_{i}\mid i\in\mathbb{I}_{2r+2}\}$, and $\uV$ be the $\bbq(q)$-vector space  spanned by $\{\eta_{j}\mid j\in\mathbb{I}_{2r+4}\}$, the latter of which is a two-dimensional extension of the former. Furthermore, there is a natural representation of $\buq(\sll_{2r+4})$ on $\uV$ via	
\begin{equation}
\begin{alignedat}{5}\label{action1.1}
	&F_{i}\eta_{i-\frac{1}{2}}=\eta_{i+\frac{1}{2}},&&
	E_{i}\eta_{i+\frac{1}{2}}=\eta_{i-\frac{1}{2}},\\
	&F_{i}\eta_{j}=0,&&E_{i}\eta_{j+1}=0,&&(j\neq i-\frac{1}{2}),\\
	&K_{i}\eta_{i-\frac{1}{2}}=q\eta_{i-\frac{1}{2}},&\quad&K_{i}\eta_{i+\frac{1}{2}}
	=q^{-1}\eta_{i+\frac{1}{2}},&&\\
	&K_{i}\eta_{j}=\eta_{j}&& &&(j\neq i\pm\frac{1}{2}).
\end{alignedat}
\end{equation}
Consequently,
\begin{equation}
\begin{alignedat}{3}\label{action1.2}
	&k_{i}\eta_{i-\frac{1}{2}}=q\eta_{i-\frac{1}{2}},&&
	k_{i}\eta_{-(i-\frac{1}{2})}
	=q\eta_{-(i-\frac{1}{2})},\cr
	&k_{i}\eta_{i+\frac{1}{2}}=q^{-1}\eta_{i+\frac{1}{2}},&&
	k_{i}\eta_{-(i+\frac{1}{2})}=q^{-1}\eta_{-(i+\frac{1}{2})},\cr
	&k_{i}\eta_{j}=\eta_{j}\quad\text{for }j\neq\pm(i\pm\frac{1}{2}).
\end{alignedat}
\end{equation}
Hence,  there is a $\buq^{\imath}(\sll_{2r+4})$-module structure on $\underline{V}$.
Let $m$ be a fixed integer not less than  2. It is clear
that $\underline{V}^{\otimes m}$ is also a $\textbf{U}_{q}(\mathfrak{s}\mathfrak{l}_{2r+4})$-module in a usual way. Precisely,  the action of $u\in\textbf{U}_{q}(\mathfrak{s}\mathfrak{l}_{2r+4})$ on $\underline{V}^{\otimes m}$
is defined as
$$\Delta^{m}(u):=(\Delta\otimes\id^{\otimes m-2})\circ\cdots\circ(\Delta\otimes\id)\circ\Delta(u).$$
Hereinafter we denote  $\overset{s}{\overbrace{\square \otimes\cdots\otimes \square}}$ by  $\square^{\otimes s}$  for an operator or a vector $\square$.
Then we have
\begin{equation}
\begin{aligned}\label{action2}
	&\Delta^{m}(K_{i})=K_i^{\otimes m}; \\
	&\Delta^{m}(E_{i})=\sum_{j=1}^m
	\id^{\otimes (j-1)}\otimes E_{i}\otimes  K_{i}^{\otimes (j-m)};\\
	&\Delta^{m}(F_{i})=\sum_{j=1}^m
	{K_{i}}^{\otimes(j-1)}\otimes F_{i}\otimes \id^{\otimes(m-j)},
\end{aligned}
\end{equation}
where  $K_i^{\otimes{(j-m)}}$ means $(K_i^{-1})^{\otimes (m-j)}$.

\subsection{Hecke algebras of type $\tsb$ }
Denote by $\mathbb{I}_{2r+4}^{m}$ the set of  $m$-tuples from $\mathbb{I}_{2r+4}$. This means
$$\mathbb{I}_{2r+4}^{m}:=\{(a_{1,}a_{2},\ldots,a_{m})\mid a_{i}\in\mathbb{I}_{2r+4}\text{ with } i=1,2,\ldots,m\}.$$
View $f\in\mathbb{I}_{2r+4}^{m}$ as a function $f:\{1,\ldots,m\}\rightarrow\mathbb{I}_{2r+4}$.
For any $f\in\mathbb{I}_{2r+4}^{m}$, we can write $f=(f(1),\ldots,f(m))$, and  define $M_{f}:=\eta_{f(1)}\otimes\eta_{f(2)}\otimes\cdots \otimes\eta_{f(m)}$.
Then the set $\{M_{f}\mid f\in\mathbb{I}_{2r+4}^{m}\}$ forms a basis for $\uV^{\otimes m}$.
Let $\scrW(\tsb_{m})$ be the Coxeter group of type $\tsb_{m}$ with simple
reflections $s_{j}$, $0\le j\le m-1$, while the subgroup generated by all
$s_{i}$ with $1\leq i\leq m-1$ is exactly $\scrW(\tsa_{m-1})$ coinciding with $\fsm$.

Let $\calh(\tsb_{m})$ be the Iwahori-Hecke algebra of type $\tsb_{m}$ over $\bbq(q)$. By definition, $\calh(\tsb_m)$  is generated by $H_{0}$, $H_{1}$, \ldots, $H_{m-1}$ as a $\bbq(q)$-algebra,
subject to the following relations,
\begin{alignat*}{3}
	&(H_{i}-q^{-1})(H_{i}+q)=0, &\text{for }& i\geq0;\\
	&H_{i}H_{i+1}H_{i}=H_{i+1}H_{i}H_{i+1}, &\text{for }& 0\leq i\leq m-2;\\
	&H_{i}H_{j}=H_{j}H_{i}, &\text{for }& |i-j|>1;\\
	&H_{0}H_{1}H_{0}H_{1}=H_{1}H_{0}H_{1}H_{0}.&&
\end{alignat*}
For  $\sigma\in \scrW(\tsb_{m})$ with a reduced expression $\sigma=s_{i_{1}}\cdot\cdot\cdot s_{i_{k}}$, the corresponding element  $H_{\sigma}:= H_{i_{1}}\cdot\cdot\cdot H_{i_{k}}\in\calh(\tsb_m)$ is well-defined, due to Matsumoto's theorem (see for example \cite{Geck}).

\subsubsection{} The group $\scrW(\tsb_{m})$ and its subgroup $\fsm$  naturally act  by transposition  on $\mathbb{I}_{2r+4}^{m}$ (from the right side):
for $0\le j\le m-1$ and any $f\in\mathbb{I}_{2r+4}^{m}$,
\begin{align*}
	f\cdot s_{j}=\begin{cases}
		(\ldots,f(j+1),f(j),\ldots), & \text{ if } j>0,\cr
		(-f(1),f(2),\ldots,f(m)), & \text{ if } j=0.
	\end{cases}
\end{align*}
\subsubsection{}
By a straightforward computation, there is correspondingly a right action of the Hecke algebra $\calh(\tsb_{m})$ on the $\bbq(q)$-vector space $\uV^{\otimes m}$ as below:
\begin{align}\label{eq: 1.2.2}
	M_{f}H_{i}=\begin{cases}
		q^{-1}M_{f}, & \text{ if } i\neq0 \text{  and } f(i)=f(i+1);\cr
		M_{fs_{i}}, & \text{ if } i\neq0 \text{  and } f(i)<f(i+1);\cr
		M_{fs_{i}}+(q^{-1}-q)M_{f}, & \text{ if } i\neq0 \text{ and } f(i)>f(i+1);\\
		M_{fs_{0}}, & \text{ if } i=0 \text{  and } f(1)>0;\cr
		M_{fs_{0}}+(q^{-1}-q)M_{f}, & \text{ if } i=0 \text{ and } f(1)<0.
	\end{cases}
\end{align}

\subsection{}\label{sec: 1.4} Recall the arguments in \S\ref{natural rep}.
The left action of $\textbf{U}_{q}(\mathfrak{s}\mathfrak{l}_{2r+4})$ restricts to $\textbf{U}_{q}^{\imath}(\mathfrak{s}\mathfrak{l}_{2r+4})$. Therefore $\textbf{U}_{q}^{\imath}(\mathfrak{s}\mathfrak{l}_{2r+4})$
and $\calh(\tsb_{m})$ have  left and right action on $\underline{V}^{\otimes m}$ respectively,
which we denote by $\Phi$ and $\Psi$ respectively. On the other hand, it is worthwhile reminding that
$\calh(\tsb_{l})$ has a right action on $V^{\otimes l}$. We will denote this
action of $\calh(\tsb_{l})$ on $V^{\otimes l}$ by $\Psi_{l}^{V}$.

The following is the remarkable $\imath$Schur duality established  by Bao-Wang.

\begin{theorem}\label{iduality} \rm{(\cite[\S 5.4]{BW})} Both $\Phi(\mathrm{\mathbf{U}}_q^\imath(\mathfrak{sl}_{2r+4}))$
	and $\Psi(\calh(\tsb_{m}))$ commute with each other, and they form double
	centralizers, i.e.,
	\[
	\End_{\mathrm{\mathbf{U}}_q^\imath(\mathfrak{sl}_{2r+4})}(\uV^{\otimes m})^{\mathrm{op}}=\Psi(\calh(\tsb_{m})),
	\]
	\[
	\End_{\calh(\tsb_{m})}(\underline{V}^{\otimes m})=\Phi(\mathrm{\mathbf{U}}_q^\imath(\mathfrak{sl}_{2r+4})).
	\]
\end{theorem}

\vskip5pt
\section{Duplex Hecke algebras}
In this section, we will introduce duplex Hecke algebras (of type
$\tsb$). This concept originates from degenerate duplex Hecke algebras
in \cite{SXY}, extending the notion of duplex Hecke algebras of type $\tsa$ in \cite{XZ}.

Recall that the symmetric group  $\fsl$ of $l$ letters  ($l\geq2$) has a set of canonical generators $\{s_{1}, \ldots, s_{l-1}\}$ with relations $s_i^2=\id$ for $i=1,\ldots,l-1$,  $s_is_j=s_js_i$ if  $i,j=1,\ldots,l-1$ satisfying $|i-j|>1$, and $s_is_{i+1}s_i=s_{i+1}s_is_{i+1}$ for $i=1,\ldots,l-2$.
Moreover, we appoint that $\frak{S}_1=\frak{S}_{0}=\{\id\}$, the identity group,  for convenience.

\subsection{Basic Definition}
\begin{defn}\label{defn: duplex H}
	A duplex Hecke algebra $\dha_m$	is a $\mathbb{Q}(q)$-algebra generated
	by  all $T_{i}$, with $i=0,1,\ldots\ ,m-1$, and all $x_{\sigma}^{(l)}$, with $\sigma\in \fsl$ for $0\protect\leq l\protect\leq m$,
	subject to the follow relations:
	
\begin{equation}
	\begin{alignedat}{3}\label{dha def1}
	&(T_{i}+q)(T_{i}-q^{-1})=0,\quad &\text{for }& 0\leq i\leq m-1;\\
	&T_{i}T_{i+1}T_{i}=T_{i+1}T_{i}T_{i+1}, &\text{for }& 0\leq i\leq m-2;\\
	&T_{i}T_{j}=T_{j}T_{i}, &\text{for }& |i-j|>1;\\
	&T_{0}T_{1}T_{0}T_{1}=T_{1}T_{0}T_{1}T_{0}.&&
    \end{alignedat}
\end{equation}
	and
	\begin{equation}
	\begin{aligned}\label{dha def2}
		x_{\sigma}^{(l)}x_{s_{i}}^{(l)}&=\begin{cases}
			x_{\sigma s_{i}}^{(l)}, & l(\sigma s_{i})=l(\sigma)+1;\cr
			x_{\sigma s_{i}}^{(l)}+(q^{-1}-q)x_{\sigma}^{(l)}, & l(\sigma s_{i})=l(\sigma)-1;
		\end{cases}
		\cr	
		x_{s_{i}}^{(l)}x_{\sigma}^{(l)}&=\begin{cases}
			x_{s_{i}\sigma}^{(l)}, & l(s_{i}\sigma )=l(\sigma)+1;\\
			x_{s_{i}\sigma}^{(l)}+(q^{-1}-q)x_{\sigma}^{(l)}, & l(s_{i}\sigma )=l(\sigma)-1;
		\end{cases}\\
        x_{\sigma}^{(l)}x_{\gamma}^{(k)}&=0,\quad l\neq k,
	\end{aligned}
\end{equation}
	along with
\begin{equation}
	\begin{alignedat}{3}\label{dha def3}
		T_{i}x_{\sigma}^{(l)}&=x_{s_{i}}^{(l)}x_{\sigma}^{(l)},&\;0<i<l;\\
		x_{\sigma}^{(l)}T_{i}&=x_{\sigma}^{(l)}x_{s_{i}}^{(l)},&\;0<i<l;\\
		T_{i}x_{\sigma}^{(l)}&=q^{-1}x_{\sigma}^{(l)}=x_{\sigma}^{(l)}T_{i},&i>l.
	\end{alignedat}
\end{equation}
\end{defn}

\begin{remark}\label{rem: 2.2}
     From the formula (\ref{dha def1}), $\calh(\tsb_m)$ can be imbedded into  $\dha_m$ as a subalgebra, and we can regard $T_i$ as an element in $\calh(\tsb_m)$. In particular, all $T_i$ with $i=0,1,\ldots,m-1$ are invertible.

\end{remark}

\subsection{Canonical representation of $\dha_m$ on ${\uV}^{\otimes m}$}\label{subsection2.2}
 Keep the notations as before. In particular, we keep the notations in \S\ref{natural rep}.  In order to endow an $\dha_m$-module structure on  ${\uV}^{\otimes m}$, we first introduce a decomposition of ${\uV}^{\otimes m}$.
Let $n$ be a nonnegative integer. Set $\underline{n}:=\{1,2,\cdots,n-1,n\}$
for $n>0$ and $\underline{0}=\emptyset$, where $\emptyset$
denotes the empty set.

For $I\subset\underline{m}$ with $\#I=l$, and  $J\subset\underline{m}\backslash I$,
 we set $\underline{V}_{I,J}^{\otimes m}$ to be the subspace of ${\uV}^{\otimes m}$ spanned by all $M_{f}:=\eta_{f(1)}\otimes\eta_{f(2)}\otimes\cdots\otimes
\eta_{f(m)}$ with $f(k)\in\mathbb{I}_{2r+2}$ for $k\in I$, $f(k)=-r-\frac{3}{2}$ for $k\in J$, and $f(k)=r+\frac{3}{2}$ for  $k\in\underline{m}\backslash (I\cup J)$, and let
$$\underline{V}_{l}^{\otimes m}=\underset{\#I=l,J\subset\underline{m}\backslash I}{\bigoplus}\underline{V}_{I,J}^{\otimes m}.$$
 In the set-up above, ${\uV}^{\otimes m}$ can be decomposed into a direct sum of subspaces
$${\uV}^{\otimes m}=\bigoplus_{l=0}^m \uV_{l}^{\otimes m}=\bigoplus_{I,J\subset\underline{m}, I\cap J=\emptyset}\uV_{I,J}^{\otimes m}. $$
From now on, we suppose the integer $l\in \{0,1,\ldots,m\}$. Notice that we have  $\underline{V}_{\underline{l}, \emptyset}^{\otimes m}=V^{\otimes l}\otimes\eta_{r+\frac{3}{2}}^{\otimes m-l}$.

\subsubsection{}\label{sec: 2.2.1}
Keep the notations from \S\ref{sec: 1.4}. For $I\subset\underline{m}$
with $\#I=l,$ $J\subset\underline{m}\backslash I$, there exists a linear transformation $\omega_{I,J}$ mapping $\underline{V}_{I,J}^{\otimes m}$ onto $\uV_{\underline{l},\emptyset}^{\otimes m}$.
 Actually, this $\omega_{I,J}$ can be expressed as a composition  of $T_{i}^{\pm1}$ with $0\leq i\leq m-1$ and acts on $\underline{V}_{I,J}^{\otimes m}$ by natural representation $\Psi$ of Hecke algebra $\calh(\tsb_m)$.
We can regard it as an element in $\dha_m$ (see Remark \ref{rem: 2.2}). We explain this operation in the following lemma.

\begin{lemma}\label{lem: 3.2} For $I\subset\underline{m}$
	with $\#I=l,$ $J\subset\underline{m}\backslash I$, there exists an
	element $\omega_{I,J}$ in $\dha_m$ generated by some $T_{i}^{\pm1}$ with $i\in\{0,1,\ldots,m-1\}$
	such that $\underline{V}_{\underline{l},\emptyset}^{\otimes m}=\Psi(\omega_{I,J})(\underline{V}_{I,J}^{\otimes m})$.
\end{lemma}

\begin{proof}
    We construct the desired element by induction on $\#J$. 
    
     In the case when  $\#J=0$, i.e. $J=\emptyset$, we take $r$ to be the minimal element of $\underline{m}\backslash I$. If further, $I=\underline{l}$, then we have  nothing to do.  Now we suppose $I \neq \underline{l}$. Take the minimal element $i$ in $I$ with $i>r$.
    For any $\eta_{f(1)}\otimes\eta_{f(2)}\otimes\cdots\otimes
	\eta_{f(m)}\in \uV_{I,\emptyset}^{\otimes m}$, we have $f(k)>f(i)$ for any $r\leq k<i$ . Then by applying  $\Psi$,  
	  we have
	$$\Psi(T_{i-1}^{-1}T_{i-2}^{-1}\cdots T_{r}^{-1})(\uVm_{I,\emptyset})=\Psi(T_{r}^{-1})\Psi(T_{r+1}^{-1})\cdots\Psi(T_{i-1}^{-1})(\uVm_{I,\emptyset})=\uVm_{I',\emptyset},$$
	where $I'=(I\backslash \{i\})\cup \{r\}$. Applying operation repeatedly in the case $I'$, we can change index $I$ into $\underline{l}$, which shows $\uVm_{\underline{l},\emptyset}=\Psi(\omega_{I,\emptyset})(\uVm_{I,J})$ for some $\omega_{I,\emptyset}$ as desired.
	
	Next, consider the case when $\#J>0$.   Then there exists a minimal element $j\in J$, i.e, all $\eta_{f(1)}\otimes\eta_{f(2)}\otimes\cdots\otimes
	\eta_{f(m)}\in \uV_{I,J}^{\otimes m}$ satisfy $f(i)>f(j)$ for any $i<j$. Then by applying the action $\Psi$, we have
	$$\Psi(T_{j-1}^{-1}T_{j-2}^{-1}\cdots T_{1}^{-1}T_{0})(\underline{V}_{I,J}^{\otimes m})=\Psi(T_{0})\Psi(T_{1}^{-1})\cdots\Psi(T_{j-1}^{-1})(\uV_{I,J}^{\otimes m})=\uV_{I',J'}^{\otimes m},$$
	where $I'$ and $J'$ is some subset of $\underline{m}$ with $\#J'=\#J-1$. We can now apply induction to $\uV_{I',J'}^{\otimes m}$ such that  $\underline{V}_{\underline{l},\emptyset}^{\otimes m}=\Psi(\omega_{I',J'})(\underline{V}_{I',J'}^{\otimes m})$ for some $\omega_{I',J'}$. Thus $\omega_{I,J}=T_{j-1}^{-1}T_{j-2}^{-1}\cdots T_{1}^{-1}T_{0}\omega_{I',J'}$ as desired.
\end{proof}

\subsubsection{}\label{sec action of h}

Note that ${\uV}^{\otimes m}=\bigoplus_{l=0}^m \uV_{l}^{\otimes m}=\bigoplus_{I,J\subset\underline{m}, I\cap J=\emptyset}\uV_{I,J}^{\otimes m}.$ By using Lemma \ref{lem: 3.2}, we can only consider the actions on $\uV_{\underline{l},\emptyset}$ and then extend $\Psi$. We can obtain an $\dha_m$-module structure on $\uVm$. Such an extension seems natural by the structure of duplex Hecke algebras. Recall that $\Psi_l^V$ in $\S\ref{sec: 1.4}$ denotes the natural representation of $\calh(\tsb_{l})$ on $V^{\otimes l}$.

\begin{prop}\label{prop1}
	There exists a unique $\mathbb{Q}(q)$-algebra
	anti-homomorphism $\Xi:\dha_m\rightarrow \End(\uVm)$
	such that
		\begin{equation}
	\begin{aligned}\label{eq: defining Xi}
		&\Xi(T_{i})=\Psi(H_{i}),&&\text{for } i=0,1,2,\ldots\ ,m-1;\\
		&\Xi(x_{\sigma}^{(l)})|_{\underline{V}_{\underline{l},\emptyset}^{\otimes m}}=\Psi_{l}^{V}(H_{\sigma})\otimes \id^{\otimes m-l}, &&\text{for } \sigma\in \fsl,0\leq l\leq m;\\
		&\Xi(x_{\sigma}^{(l)})(M_{f})=0, &&\text{for } M_{f}\notin\underline{V}_{\underline{l},
			\emptyset}^{\otimes m}.
	\end{aligned}
	\end{equation}
\end{prop}
\begin{proof}

Since the action of generators of $\dha_m$ are determined by formula (\ref{eq: defining Xi}), the uniqueness is obvious if this homomorphism exists.
It is enough to show that the right action $\Xi$  of $\dha_m$
on $\underline{V}^{\otimes m}$
keeps  relations
(\ref{dha def1})-(\ref{dha def3}).

 With the start point of extending $\Psi$ to $\Xi$,  $\Xi|_{\langle T_{i},i=0,1,...,m-1\rangle}$ is exactly the representation $\Psi$ of $\calh(\tsb_{m})$.  So $\Xi$ keeps relations (\ref{dha def1}).

For $0<i<l$. Recall that in this case,
the action of $H_{i}$ on $$M_{f}=\eta_{f(1)}\otimes\eta_{f(2)}\otimes\cdots\otimes\eta_{f(m)}$$
only changes the first
$l$ factors of $M_{f}$  which belongs to $\uV_{\underline{l},\emptyset}^{\otimes m}$. If $M_f\in \uV_{\underline{l},\emptyset}^{\otimes m}$,  we then have
 $$\Xi(T_{i})(M_{f})=\Psi(H_{i})(M_{f})=\Psi_{l}^{V}(H_{i})\otimes \id^{\otimes m-l}(M_{f})=\Xi(x_{s_{i}}^{(l)})(M_{f}),$$
and hence
$$\Xi(T_{i}x_{\sigma}^{(l)})(M_{f})=\Xi(x_{\sigma}^{(l)})\Xi(T_{i})(M_{f})=\Xi(x_{\sigma}^{(l)})\Xi(x_{s_{i}}^{(l)})(M_{f})=\Xi(x_{s_{i}}^{(l)}x_{\sigma}^{(l)})(M_{f}),$$
$$\Xi(x_{\sigma}^{(l)}T_{i})(M_{f})=\Xi(T_{i})\Xi(x_{\sigma}^{(l)})(M_{f})=\Xi(x_{s_{i}}^{(l)})\Xi(x_{\sigma}^{(l)})(M_{f})=\Xi(x_{\sigma}^{(l)}x_{s_{i}}^{(l)})(M_{f}).$$
If $M_f\notin\uV_{\underline{l},\emptyset}^{\otimes m}$, the both sides of equations are zero, since $\Xi(x_{\sigma}^{(l)})(M_f)=0$ and $\Xi(T_i)(M_f)\notin\uV_{\underline{l},\emptyset}$.

For $i>l$, we have
$$\Xi(T_{i}x_{\sigma}^{(l)})=q^{-1}\Xi(x_{\sigma}^{(l)})=\Xi(x_{\sigma}^{(l)}T_{i}),$$
since $\Xi(T_i)=\Psi(H_i)$ acts as multiplication by $q^{-1}$ on vector of $\uV_{\underline{l},\emptyset}^{\otimes m}$ and annihilates  vectors which are outside of  $\uV_{\underline{l},\emptyset}^{\otimes m}$.
Thus, $\Xi$ keeps relation (\ref{dha def3}).

For $l\neq k$, it follows easily from definition that $\Xi(x_{\sigma}^{(l)}x_{\gamma}^{(k)})=\Xi(x_{\gamma}^{(k)}x_{\sigma}^{(l)})=0$.
Regard $\mathfrak{S}_l$ as a subgroup of $\mathfrak{S}_m$, we get similarly for $0<i<l$
$$\Xi(x_{\sigma}^{(l)}x_{s_i}^{(l)})=\Psi({H_\sigma H_i}).$$
The relation (\ref{dha def2}) holds by the composition formula in Hecke algebra.

We conclude the existence of $\Xi$, which is a representation of $\dha_m$ on $\uVm$.

\end{proof}

\section{Levi-type $q$-Schur-Weyl duality of type $\tsb$}
In this section, we establish a duality between the Levi-type $\imath$quantum
group and the duplex Hecke algebra $\dha_m$.

\subsection{$\imath$quantum group of Levi-type}\label{Levi-type}
\begin{defn}
	Denote by $L_{q}^{\imath}(\sll_{2r+4})$  the subalgebra of $\buq^{\imath}(\mathfrak{s}\mathfrak{l}_{2r+4})$
	generated by $B_{0}:= E_{0}+qF_{0}K_{0}^{-1}+K_{0}^{-1}$, all $B_{i}:= E_{i}+F_{-i}K_{i}^{-1}$ with $0\neq i\in\mathbb{I}_{2r+1}$, and all $k_{i}:= K_{i}K_{-i}^{-1}$ with $i\in\mathbb{I}_{2r+3}$.  Call $L_{q}^{\imath}(\sll_{2r+4})$ an $\imath$quantum group of Levi-type.
\end{defn}
It is easy to see $L_{q}^{\imath}(\mathfrak{s}\mathfrak{l}_{2r+4})
\cong\textbf{U}_{q}^{\imath}(\mathfrak{s}\mathfrak{l}_{2r+2})\oplus
\langle k_{r+1},k_{r+1}^{-1}\rangle$
as a space.

\subsection{} The action of $\buq^{\imath}(\mathfrak{s}\mathfrak{l}_{2r+4})$ on $\uVm$ restricts to the subalgebra $L_{q}^{\imath}(\sll_{2r+4})$ and denoted by $\Phi$ again. On the other hand, we establish the action of $\dha_m$ on $\uVm$ in Proposition \ref{prop1}. We have the following observation, which  reveals the relation between the Levi-type $\imath$quantum group $L_q^\imath(\sll_{2r+4})$ and the duplex Hecke algebra $\dha_m$.

\begin{lemma}\label{lem commutate}
	The actions of $L_{q}^{\imath}(\sll_{2r+4})$
	and $\dha_m$ on $\uVm$
	commute with each other, i.e. $\Phi(L_{q}^{\imath}(\sll_{2r+4}))\subset \End_{\dha_m}(\uVm)$.
\end{lemma}
\begin{proof}
We only need to show the actions of $k_{i}$ with $i\in\mathbb{I}_{2r+3}$, and
$B_{i}$ with $i\in\mathbb{I}_{2r+1}$ on $\uV^{\otimes m}$ commute with the action of $T_{i}$ with $i=0,1,\ldots\ ,m-1$, and
$x_{\sigma}^{(l)}$ with $\sigma\in \fsl$ for $0\leq l\leq m$, respectively.

The mutual commutant property  of all $\Xi(T_i)$ with $\Phi(L_{q}^{\imath}(\sll_{2r+4}))$ is ensured by  $\imath$Schur duality (Theorem \ref{iduality}) and Remark \ref{rem: 2.2}.
We first check that all $\Psi(B_i)$ with $0\neq i\in\mathbb{I}_{2r+1}$ commute with all $\Xi(x_{\sigma}^{(l)})$ with $\sigma\in\fsl$. 	From (\ref{action2}), we have that
	\begin{align*}
	\Delta^{m}(B_{i}) &=\Delta^{m}(E_{i}+F_{-i}K_{i}^{-1})=\Delta^{m}(E_{i})+
	\Delta^{m}(F_{-i}K_{i}^{-1})\\
	&=\sum_{j=1}^m \id^{\otimes (j-1)}\otimes E_{i}\otimes K_{i}^{\otimes(j-m)}+(\sum_{j=1}^m {K_{-i}}^{\otimes(j-1)}\otimes F_{-i}\otimes \id^{\otimes(m-j)})
	(K_{i}^{\otimes -m})\\
	&=\sum_{j=1}^m \id^{\otimes(j-1)}\otimes E_i\otimes K_{i}^{\otimes(j-m)}+k_i^{\otimes (1-j)}\otimes F_{-i}K_i^{-1}\otimes K_{i}^{\otimes(j-m)}.
\end{align*}
Notice any element of $\uV_{\underline{l},\emptyset}$ can be written as $\alpha\otimes\eta_{r+\frac{3}{2}}^{\otimes m-l}$ with some  $\alpha\in V^{\otimes l}$. For $0\neq i\in\mathbb{I}_{2r+1}$, we can get
\begin{align*}
	&\Phi(B_{i})(\alpha\otimes\eta_{r+\frac{3}{2}}^{\otimes m-l})
	=\Delta^{m}(B_{i})(\alpha\otimes\eta_{r+\frac{3}{2}}^{\otimes m-l})\\
	=&(\sum_{j=1}^m \id^{\otimes(j-1)}\otimes E_i\otimes K_{i}^{\otimes(j-m)}+ k_i^{\otimes (1-j)}\otimes F_{-i}K_i^{-1}\otimes K_{i}^{\otimes(j-m)})(\alpha\otimes\eta_{r+\frac{3}{2}}^{\otimes m-l}).\\
	=&(\sum_{j=1}^l \id^{\otimes(j-1)}\otimes E_i\otimes K_{i}^{\otimes(j-m)} +k_i^{\otimes (1-j)}\otimes F_{-i}K_i^{-1}\otimes  K_{i}^{\otimes(j-m)})(\alpha\otimes\eta_{r+\frac{3}{2}}^{\otimes m-l})\\
	=&(\sum_{j=1}^l \id^{\otimes (j-1)}\otimes E_i+k_i^{\otimes(1-j)}\otimes F_{-i}K_i^{-1})(\alpha)\otimes  \eta_{r+\frac{3}{2}}^{\otimes m-l}\cr
	=&	\Delta^{l}(B_{i})(\alpha)\otimes\eta_{r+\frac{3}{2}}^{\otimes m-l}.
\end{align*}
The third and fourth equalities follow from the fact
\begin{equation}
\begin{aligned}\label{compu1}
	E_{i}\eta_{r+\frac{3}{2}}=F_{-i}K_{i}^{-1}\eta_{r+\frac{3}{2}}=0, \text{ and } K_{i}^{-1}\eta_{r+\frac{3}{2}}=\eta_{r+\frac{3}{2}} \text{ for } i\in\mathbb{I}_{2r+1}.	
\end{aligned}
\end{equation}
Hence, from $\imath$Schur duality on $V^{\otimes l}$, we have
\begin{align*}
	\Xi(x_{\sigma}^{(l)})\Phi(B_{i})(\alpha\otimes\eta_{r+\frac{3}{2}}^{\otimes m-l})
	=&\Xi(x_{\sigma}^{(l)})(\Delta^{l}(B_{i})(\alpha)\otimes\eta_{r+\frac{3}{2}}^{\otimes m-l})\cr
	=&\Psi_{l}^{V}(H_{\sigma})\Delta^{l}(B_{i})(\alpha)\otimes\eta_{r+\frac{3}{2}}^{\otimes m-l}\cr
	=&\Delta^{l}(B_{i})\Psi_{l}^{V}(H_{\sigma})(\alpha)\otimes\eta_{r+\frac{3}{2}}^{\otimes m-l}\cr
	=&\Phi(B_{i})\Xi(x_{\sigma}^{(l)})(\alpha\otimes\eta_{r+\frac{3}{2}}^{\otimes m-l}).
\end{align*}

Now we consider the action on the other parts outside of $\uV_{\underline{l},\emptyset}$. We already have $\Phi(B_{i})\Xi(x_{\sigma}^{(l)})=0$. For any
$M_{f}=\eta_{f(1)}\otimes\eta_{f(2)}\otimes\cdots\otimes\eta_{f(m)}\notin\underline{V}_{\underline{l},
	\emptyset}^{\otimes m}$,
there are two possibilities:
\begin{itemize}
	\item[(i)] There exists $k\in\{1,2,\ldots,l\}$ such that $f(k)\in\{-r-\frac{3}{2},r+\frac{3}{2}\}$;
	\item[(ii)] There exists  $k\in\{l+1,l+2,\ldots,m\}$ with $f(k)\neq r+\frac{3}{2}$.
\end{itemize}

In the first case, following from formula (\ref{compu1}), the $k$th factor of each term of $\Psi(B_i)(M_f)$ does not lie in $\uV_{\underline{l},\emptyset}$ for $j\neq k$ and is zero for $j=k$.
It shows $\Xi(x_{\sigma}^{(l)})\Phi(B_{i})=0$ as required.

For the second case, the action of $E_{i}$, $F_{-i}K_{i}^{-1}$,  $K_{i}^{-1}$, and $k_i^{-1}$ for $i\in\mathbb{I}_{2r+1}$ cannot change $\eta_{k}$ to $\eta_{r+\frac{3}{2}}$. We still have $\Xi(x_{\sigma}^{(l)})\Phi(B_{i})=0$.

It remains to show $\Phi(B_0)$ commutes with $\Xi(x_{\sigma}^{(l)})$. Notice that
	\begin{align*}
    &\Delta^{m}(B_{0})
    =\Delta^{m}(E_{0})+
	\Delta^{m}(qF_{0}K_{0}^{-1})+\Delta^m({K_0^{-1}})\\
	=&\sum_{j=1}^m \id^{\otimes (j-1)}\otimes E_{0}\otimes K_{i}^{\otimes(j-m)}+q(\sum_{j=1}^m {K_{-i}}^{\otimes(j-1^)}\otimes F_{-i}\otimes \id^{\otimes(m-j)})
	(K_{i}^{\otimes -m})+K_0^{\otimes -m}\\
	=&\sum_{j=1}^m \id^{\otimes(j-1)}\otimes E_i\otimes K_{i}^{\otimes(j-m)}+q k_i^{\otimes (1-j)}\otimes F_{-i}K_i^{-1}\otimes K_{i}^{\otimes(j-m)}+K_0^{\otimes -m}.
\end{align*}
Similarly, we still have
\begin{align*}
	\Xi(x_{\sigma}^{(l)})\Phi(B_{0})(\alpha\otimes\eta_{r+\frac{3}{2}}^{\otimes m-l})
	=&\Psi_{l}^{V}(H_{\sigma})\Delta^{l}(B_{0})(\alpha)\otimes\eta_{r+\frac{3}{2}}^{\otimes m-l}\\
	=&\Delta^{l}(B_{0})\Psi_{l}^{V}(H_{\sigma})(\alpha)\otimes\eta_{r+\frac{3}{2}}^{\otimes m-l}\\
	=&\Phi(B_{0})\Xi(x_{\sigma}^{(l)})(\alpha\otimes\eta_{r+\frac{3}{2}}^{\otimes m-l}),
\end{align*}
and
 $$\Xi(x_{\sigma}^{(l)})\Phi(B_{0})(M_f)=\Phi(B_{0})\Xi(x_{\sigma}^{(l)})(M_f)=0,$$
 for any $M_f$ outside of $\uV_{\underline{l},\emptyset}$,
since $K_0\eta_{r+\frac{3}{2}}=\eta_{r+\frac{3}{2}}$. The Lemma follows.

\end{proof}

\begin{remark}\label{embedded}
	Let $\Phi_l^{V}$ be the natural representation of $ \textbf{U}_{q}^{\imath}(\mathfrak{s}\mathfrak{l}_{2r+2})$ on $V^{\otimes l}.$
As in the proof of Lemma \ref{lem commutate}, we have
	$$\Phi(B_{i})(\alpha\otimes\eta_{r+\frac{3}{2}}^{\otimes m-l})
	=\Delta^{m}(B_{i})(\alpha\otimes\eta_{r+\frac{3}{2}}^{\otimes m-l})
	=\Delta^{l}(B_{i})(\alpha)\otimes\eta_{r+\frac{3}{2}}^{\otimes m-l}
	=\Phi_{l}^{V}(B_{i})(\alpha)\otimes\eta_{r+\frac{3}{2}}^{\otimes m-l},$$	
	$$\Phi(k_i)(\alpha\otimes\eta_{r+\frac{3}{2}}^{\otimes m-l})=\Delta^m(k_i)(\alpha\otimes\eta_{r+\frac{3}{2}}^{\otimes m-l})=\Delta^l(k_i)(\alpha)\otimes\eta_{r+\frac{3}{2}}^{\otimes m-l}=\Phi_{l}^{V}(k_i)(\alpha)\otimes\eta_{r+\frac{3}{2}}^{\otimes m-l},$$
for all $i\in\mathbb{I}_{2r+1}$. 	
  Hence, we can regard
$\Phi_{l}^{V}(\textbf{U}_{q}^{\imath}(\mathfrak{s}\mathfrak{l}_{2r+2}))$ as a subspace of
$\Phi(L_{q}^{\imath}(\mathfrak{s}\mathfrak{l}_{2r+2}))|_{\underline{V}_{\underline{l},\emptyset}^{\otimes m}}.$

\end{remark}

\subsection{$q$-Schur algebras of type $\tsb$ and their gradations} We recall some general facts about $q$-Schur algebras associated with Hecke algebras of type $\tsb$.
 For more details, one can consult  \cite{Gr} or \cite{LNX}.

For $M_{f}=\eta_{f(1)}
\otimes\eta_{f(2)}\otimes\cdots\otimes\eta_{f(m)} \in V$, the weight $ \lambda(M_{f}) $ of $ M_{f} $ is defined to be the $ (r+1) $-tuple
$$         (\lambda_{1},\cdots,\lambda_{r+1 }),       $$
where $ \lambda_{i} $ is the total number of occurrences of the indices $ i- \frac{1}{2}$ and $  -(i- \frac{1}{2}) $ in $ M_{f}$. The set of all weights will be denoted by $\Lambda_{\tsb}(2r+2,m)$.

The space $ V_{\lambda} $ is the span of all the basis vectors $ M_{f} $ with weight $ \lambda$, the vector $ \mathfrak{M}_{\lambda} $ is the unique basis vector $ M_{f} \in V_{\lambda} $ whose indices satisfy $ f(1)\le f(2)\le \cdots \le f(m) $.

The parabolic subgroup $ P_{\lambda} $ of  $\scrW(\tsa_{m-1})=\left\langle s_{1},\cdots,s_{m-1} \right\rangle$ is generated by some $ s_{i} \in \scrW(\tsa_{m-1})  $, which satisfy $ s_{i} \in P_{\lambda}$ if and only if the $i$th and $(i+1)$th entries of $  \mathfrak{M}_{\lambda} $ are equal.

The elements $ x_{\lambda} $ of $ \calh(\tsb_{m}) $   is defined by
$$         x_{\lambda}=\sum_{w\in P_\lambda} H_{w}.    $$

Following  \cite{Gr}, we give the following definition.
\begin{definition}
	The $q$-Schur algebra of type $ \tsb_{m}  $ is given by
	$$ S_{q} ^{\tsb}(2r+2,m)=\End_{\calh(\tsb_{m})}(\bigoplus_{\lambda\in \Lambda_\tsb(2r+2,m)}x_{\lambda}\calh(\tsb_{m})). $$
\end{definition}

\begin{proposition}\label{prop: Schur and gr} Keep the notations as above. The following statements hold.
	\begin{itemize}
		\item[(1)] $ V_{\lambda}\cong x_{\lambda}\calh({\tsb_{m}})$;
		\item[(2)] $ V^{\otimes m}\cong \bigoplus_{\lambda\in \Lambda_{\tsb}(2r+2,m)}x_{\lambda}\calh(\tsb_{m})$;
		\item[(3)] $S_{q} ^{\tsb}(2r+2,m)$ is a free  $\mathbb{Q}(q)$-module. Moreover,  there is a $ \bbq(q) $-basis in $S_q^\tsb(2r+2,m)$:
		$\{ \phi^{d}_{\lambda,\mu} \mid \lambda,\mu\in \Lambda_{\tsb}(2r+2,m)\}$ with $d$ is a distinguished double $P_{\lambda}-P_{\mu} $ coset representative, such that  each  $\phi^{d}_{\lambda,\mu} $ only has nonzero image on $ V_{\mu}$ ;
		\item[(4)] If $ \mu=(\mu_{1},\mu_{2},\cdots,\mu_{r+1}) $, then $ \phi^{d}_{\lambda,\mu} $  only has nonzero image on $ V^{\otimes m}_{k}  $ where $ k=m-\mu_{r+1} $.
	\end{itemize}
	
\end{proposition}
\begin{proof} Parts (1)-(3) follow from  \cite[\S4]{Gr}. As to  (4), we remark that  $\phi^{d}_{\lambda,\mu} $ only has nonzero image on $ V_{\mu}$ and $V_{\mu}\subset V^{\otimes m}_{k} $, where $ k=m-\mu_{r+1}$ (with the notation from \S\ref{subsection2.2}). Then (4) follows immediately.
\end{proof}

\subsection{} Now we can establish one part of the property of double centralizers.
\begin{theorem}\label{thm: main thm1}
	If $q$ is transcendental
	over $\mathbb{Q}$. We have $\Phi(L_{q}^{\imath}(\mathfrak{s}\mathfrak{l}_{2r+4}))
	=\End_{\dha_m}(\underline{V}^{\otimes m}).$
\end{theorem}

\begin{proof}
	By Lemma \ref{lem commutate}, it is enough to  show $\End_{\dha_m}(\underline{V}^{\otimes m})\subset\Phi(L_{q}^{\imath}(\mathfrak{s}\mathfrak{l}_{2r+4})).$
	
	Note that $\End_{\dha_m}(\underline{V}^{\otimes m})\subset \End_{\calh(\tsb_{m})}(\underline{V}^{\otimes m})=S_q^\tsb(2r+4,m)$.
    Thanks to Proposition \ref{prop: Schur and gr}, for every $\phi\in \End_{\dha_m}(\underline{V}^{\otimes m})$,
	we have $\phi=\underset{\lambda,\mu \in \Lambda_{\tsb}(2r+4,m)}{\sum}a_{\lambda\mu d}\phi^{d}_{\lambda,\mu}$, where $a_{\lambda\mu d}\in\mathbb{Q}(q)$. More precisely,	
	$$\phi= \sum_{l=0}^{m}\phi_{l}=\displaystyle\sum_{l=0}^m\sum_{\lambda,\mu \in \Lambda_{\tsb}(2r+4,m),\mu_{r+2}=m-l}a_{\lambda\mu d}\phi^{d}_{\lambda,\mu}.$$
	  We have $ \phi|_{\underline{V}_l^{\otimes m}}=\phi_{l} $ by Proposition \ref{prop: Schur and gr}(4). Now we will proceed the arguments by several steps.
	
	(a) We show that $\phi_{l}$ stabilizes $\underline{V}_{I,J}^{\otimes m}$ with $\#I=l$.
	
	We consider $\phi_{l}$ stabilizes $\underline{V}_{\underline{l},\emptyset}^{\otimes m}$ at first.
	Suppose $\nu_{j}\in\underline{V}_{\underline{l},\emptyset}^{\otimes m}$, and
	$\phi_{l}(\nu_{j})=\nu_{l}+\nu_{k}$ with $\nu_{l}\in\underline{V}_{\underline{l},\emptyset}^{\otimes m}$,
	$\nu_{k}\notin\underline{V}_{\underline{l},\emptyset}^{\otimes m}$.
	Denote the identity element of $\frak{S}_{l}$ by $e$, then	
	$$	\Xi(x_{e}^{(l)})(\nu_{l}+\nu_{k})=\Xi(x_{e}^{(l)})\phi_{l}(\nu_{j})=\Xi(x_{e}^{(l)})\phi(\nu_{j})=\phi\Xi(x_{e}^{(l)})(\nu_{j})=\nu_{l}+\nu_{k}.	$$	
	By the definition of $\Xi$, we have $\Xi(x_{e}^{(l)})(\nu_{l})=\nu_{l}$, and
	$\Xi(x_{e}^{(l)})(\nu_{k})=0$.
	It yields that $\nu_{k}=0$, and hence $\phi_{l}$ stabilizes
	$\underline{V}_{\underline{l},\emptyset}^{\otimes m}.$
	
    Then we consider a general summand $\underline{V}_{I,J}^{\otimes m}$
	by using Lemma \ref{lem: 3.2}.  Note that
	$\phi\Xi(\omega_{I,J}^{-1})=\Xi(\omega_{I,J}^{-1})\phi$, and
	$\Xi(\omega_{I,J}^{-1})\phi(\underline{V}_{\underline{l},
		\emptyset}^{\otimes m})=\Xi(\omega_{I,J}^{-1})\phi_{l}(\underline{V}_{\underline{l},\emptyset}^{\otimes m})\subset\underline{V}_{I,J}^{\otimes m}$.
	It yields that
	$$
	\phi_{l}(\underline{V}_{I,J}^{\otimes m})=\phi_{l}\Xi(\omega_{I,J}^{-1})(\underline{V}_{\underline{l},\emptyset}^{\otimes m})=\phi\Xi(\omega_{I,J}^{-1})(\underline{V}_{\underline{l},\emptyset}^{\otimes m})=\Xi(\omega_{I,J}^{-1})\phi(\underline{V}_{\underline{l},\emptyset}^{\otimes m})\subset\underline{V}_{I,J}^{\otimes m},$$	
	as desired.	
	
	(b) We claim that
	there exists an element $g_{l}\in L_{q}^{\imath}(\mathfrak{s}\mathfrak{l}_{2r+4})$ satisfying
	$$ \phi_{l}|_{\underline{V}_{l}^{\otimes m}}=\Phi(g_{l})|_{\underline{V}_{l}^{\otimes m}}.$$

	We already have that
	\begin{align*}
		\Xi(T_{0})|_{\underline{V}_{\underline{l},\emptyset}^{\otimes m}}&=\Psi_{l}^{V}(H_{0})\otimes \id^{\otimes m-l},\\
		\Xi(x_{\sigma}^{(l)})|_{\underline{V}_{\underline{l},\emptyset}^{\otimes m}}&=\Psi_{l}^{V}(H_{\sigma})\otimes \id^{\otimes m-l}\text{ for any }
		\sigma\in \fsl.
	\end{align*}

	On the other hand, consider the $\imath$Schur duality of $\Phi_l^V$ and $\Psi_l^V$ on $\uV_{\underline{l},\emptyset}^{\otimes m}$ (extended from $V^{\otimes l}$).
	For any $\phi\in \End_{\dha_m}(\underline{V}^{\otimes m})$, we have
    $$\phi_{l}|_{\underline{V}_{\underline{l},\emptyset}^{\otimes m}}\in \End_{\calh(\tsb_{l})}(\underline{V}_{\underline{l},\emptyset}^{\otimes m})
	=
	\Phi_l^V(\textbf{U}_{q}^{\imath}(\mathfrak{s}\mathfrak{l}_{2r+2}))
	\subset\Phi(L_{q}^{\imath}(\mathfrak{s}\mathfrak{l}_{2r+4}))|_{\uV_{\underline{l},\emptyset}}.$$
	(See Remark \ref{embedded}.)
	 So there exists $g_{l}\in L_{q}^{\imath}(\mathfrak{s}\mathfrak{l}_{2r+4})$ such that
	$\phi_{l}|_{\underline{V}_{\underline{l},\emptyset}^{\otimes m}}=\Phi(g_{l})|_{\underline{V}_{\underline{l},\emptyset}^{\otimes m}}$. Consequently,  $$\phi_{l}(\underline{V}_{\underline{l},\emptyset}^{\otimes m})=\phi_{l}\Xi(\omega_{I,J})(\underline{V}_{I,J}^{\otimes m})=\Phi(g_{l})\Xi(\omega_{I,J})(\underline{V}_{I,J}^{\otimes m}).$$
	Note that $\Phi(L_{q}^{\imath}(\mathfrak{s}\mathfrak{l}_{2r+4}))\subset \End_{\dha_m}(\underline{V}^{\otimes m})$. Hence $\Phi(g_{l})$ commutes with $\Xi(\omega_{I,J})$. Thus, we have	
	\begin{align*}
		\phi_{l}(\underline{V}_{I,J}^{\otimes m})
		&=\Xi(\omega_{I,J}^{-1})\phi_{l}\Xi(\omega_{I,J})(\underline{V}_{I,J}^{\otimes m})\cr
		&=\Xi(\omega_{I,J}^{-1})\Phi(g_{l})\Xi(\omega_{I,J})(\underline{V}_{I,J}^{\otimes m})\cr
		&=\Phi(g_{l})(\underline{V}_{I,J}^{\otimes m}).
	\end{align*}
	Since every direct summands of $\uV_l^{\otimes m}$ are stabilized by $\phi_l$,
	we have $\phi_{l}|_{\underline{V}_{l}^{\otimes m}}=\Phi(g_{l})|_{\underline{V}_{l}^{\otimes m}}$.
	
	(c) We construct an element of $L_{q}^{\imath}(\mathfrak{s}\mathfrak{l}_{2r+4})$ such that its image under $\Phi$ coincides with $\phi$,  which consequently complete the proof.
	
	For this purpose, we firstly set
	$X:= k_{r+1}^{r+1}k_{r}^{r}\cdots k_{2}^{2}k_{1}\in L_q^\imath(\sll_{2r+4})$.	
	From formula (\ref{action1.2}) we have
	\begin{align*}
		&\Phi(X)(\eta_{r+\frac{3}{2}})=q^{-(r+1)}\eta_{r+\frac{3}{2}},\\
		&\Phi(X)(\eta_{-r-\frac{3}{2}})=q^{-(r+1)}\eta_{-r-\frac{3}{2}},\\
	    &\Phi(X)(\eta_{i})=q\eta_{i}, \text{ for } i\in\mathbb{I}_{2r+2}.
	\end{align*}	
	Hence
	$$
	\Phi(X)|_{\underline{V}_{l}^{\otimes m}}=q^{l-(m-l)(r+1)}\cdot \id|_{\underline{V}_{l}^{\otimes m}}=q^{l(r+2)-m(r+1)}\cdot \id|_{\underline{V}_{l}^{\otimes m}}.
	$$	
	We set $F(l):= q^{l(r+2)-m(r+1)}$ (the integers $m$ and $r$ are fixed). Since $q$ is transcendental
	over \ensuremath{\mathbb{Q}}, we have $F(l)\neq F(k)$, for $l\neq k$.
	
	With the above $X$, we secondly consider the element $$G_{l}:=\frac{(X-F(1))\cdots\widehat{(X-F(l))}\cdots(X-F(m))}{(F(l)-F(1))\cdots\widehat{(F(l)-F(l))}\cdots(F(l)-F(m))},$$
	where the hat indicates as usual that the term is dropped.
	Then we have
	$$
	\Phi(G_{l})|_{\underline{V}_{l}^{\otimes m}}=(\frac{(F(l)-F(1))\cdots\widehat{(F(l)-F(l))}\cdots(F(l)-F(m))}{(F(l)-F(1))\cdots\widehat{(F(l)-F(l))}\cdots(F(l)-F(m))})\cdot \id|_{\underline{V}_{l}^{\otimes m}}=\id|_{\underline{V}_{l}^{\otimes m}},
	$$
	and
	$$
	\Phi(G_{l})|_{\underline{V}_{k}^{\otimes m}}=(\frac{(F(k)-F(1))\cdots\widehat{(F(k)-F(l))}\cdots(F(k)-F(m))}{(F(k)-F(1))\cdots\widehat{(F(k)-F(l))}\cdots(F(k)-F(m))})\cdot \id|_{\underline{V}_{l}^{\otimes m}}=0
	$$
	for $k\neq l$.
	Thus
	\begin{align*}
		\Phi(g_{l})\Phi(G_{l})\cdot{V}_{l}^{\otimes m}&=\phi_{l}\cdot\uV_{l}^{\otimes m}; \cr
		\Phi(g_{l})\Phi(G_{l})\cdot\uV_{k}^{\otimes m}&=\phi_{l}\cdot\uV_{k}^{\otimes m}=0  \; \text{ for } l\neq k.
	\end{align*}
	Hence $\phi_{l}=\Phi(g_{l}\cdot G_{l})$. So we finally succeed in  finding a desired element
	$$ \sum_{l=0}^m g_l\cdot G_l\in L_{q}^{\imath}(\sll_{2r+4}),$$
	which satisfies
	$$\phi=\sum_{l=0}^m\phi_{l}=\Phi(\sum_{l=0}^m g_{l}\cdot G_l)\in\Phi(L_{q}^{\imath}(\sll_{2r+4})).$$
	The proof is completed.
\end{proof}

\subsection{} Now we begin to deal with the other part of the double centraliser property. Under our assumption of $q$ being transcendental over \ensuremath{\mathbb{Q}}, we have that  $V^{\otimes l}$ is a completely
reducible $\calh(\tsb_{l})$-module. Suppose $V^{\otimes l}=\underset{\pi\in I_{l}}{\bigoplus}S^\pi$,
where $S^\pi$ is irreducible $\calh(\tsb_{l})$-module and $I_{l}$
is the index set. (We are not going to introduce this index set for simplicity, see details in \cite[\S 5.5]{DJ}.)
Combining with notations in \S\ref{subsection2.2}, we have
$$\underline{V}_{\underline{l},\emptyset}^{\otimes m}
=V^{\otimes l}\otimes\eta_{r+\frac{3}{2}}^{\otimes m-l}
=\underset{\pi\in I_{l}}{\bigoplus}S^\pi\otimes\eta_{r+\frac{3}{2}}^{\otimes m-l},$$
and hence
$$\underline{V}_{I,J}^{\otimes m}=\Xi(\omega_{I,J}^{-1})(\underset{\pi\in I_{l}}{\bigoplus}S^\pi\otimes\eta_{r+\frac{3}{2}}^{\otimes m-l}).$$
So we get
\begin{align*}
	\underline{V}_{l}^{\otimes m}&=\underset{\#I=l,J\subset\underline{m}\backslash I}{\bigoplus}\underline{V}_{I,J}^{\otimes m}\cr
	&=\underset{\#I=l,J\subset\underline{m}\backslash I}{\bigoplus}\Xi(\omega_{I,J}^{-1})(\underset{\pi\in I_{l}}{\bigoplus}S^\pi\otimes\eta_{r+\frac{3}{2}}^{\otimes m-l})\cr
	&=\underset{\pi\in I_{l}}{\bigoplus}\underset{\#I=l,J\subset\underline{m}\backslash I}{\bigoplus}\Xi(\omega_{I,J}^{-1})(S^\pi\otimes\eta_{r+\frac{3}{2}}^{\otimes m-l}).
\end{align*}

Under the isomorphism $\uV_{I,J}^{\otimes m}\simeq V^{\otimes l}$, all $\Xi(\omega_{I,J}^{-1})(S^\pi\otimes\eta_{r+\frac{3}{2}}^{\otimes m-l})$ are isomorphic to $S^\pi$ as an $\mathcal{H}(B_l)$-module and we denote it as $S^\pi_{I,J}$.
We set $D^{\pi}_{l}:=\underset{\#I=l,J\subset\underline{m}\backslash I}{\bigoplus}S^\pi_{I,J}$.
The following lemma shows the property of $D^\pi_{l}$.

\begin{lemma}\label{lem: 3.3} Suppose that q is transcendental over
	$\mathbb{Q}$. Then $D^{\pi}_{l}$
	is an irreducible $\dha_m$-module.
\end{lemma}

\begin{proof}
	We first prove $D^{\pi}_{l}$ is an
	$\dha_m$-module.
    It is sufficient to show that $D^\pi_l$ is  stabilized by $T_i$, with $i=0, 1, \cdots, m-1$, and $x_\sigma^{(l)}$, with $\sigma\in \fsl$ for $0\protect\leq l\protect\leq m$.

     From the proof of Lemma \ref{lem: 3.2}, $\Xi(\omega_{I,J}^{-1})$ only changes the positions of some tensor factors and takes some $\eta_{r+\frac{3}{2}}$ to $\eta_{-r-\frac{3}{2}}$.      	
	The action of $T_0$ preserves $S^\pi_{I,J}$ if $1\in I$ and takes $S^\pi_{I,J}$ to $S^\pi_{I,J'}$ for some subset $J'$ if $1\notin I$.	
	Now consider $T_i$ with $i>0$. Note that $-r-\frac{3}{2}<k<r+\frac{3}{2}$ for any $k\in\mathbb{I}_{2r+2}$. If both $i$ and $i+1$  belong to $I, J$, or $\underline{m}\backslash(I\cap J)$, the action of $T_i$ preserves  $S^\pi_{I,J}$. Otherwise,  the action of $T_i$ or $T_i^{-1}$ on $S^\pi_{I,J}$ takes $S^\pi_{I,J}$ to $S^\pi_{I',J'}$ for some subsets $I'$ and $J'$ with $\#I'=\#I=l$.	
    In summary, $D^{\pi}_{l}$ is stable under the actions of all $T_i$.

    By definition, $\Xi(x_{\sigma}^{(k)})(D^{\pi}_{l})=0$ for $k\neq l$
	and $\Xi(x_{\sigma}^{(l)})$ only has nonzero image on $S^\pi\otimes\eta_{r+\frac{3}{2}}^{\otimes m-l}$.	
     As in the proof of Lemma \ref{lem commutate},
	$\Xi(x_{\sigma}^{(l)})(S^\pi\otimes\eta_{r+\frac{3}{2}}^{\otimes m-l})=(\Psi_{l}^{V}(H_{\sigma})S^\pi)\otimes\eta_{r+\frac{3}{2}}^{\otimes m-l}$ for irreducible $\calh(\tsb_l)$-module $S^\pi$. It shows that $x_{\sigma}^{(l)}$
	stabilizes $D^{\pi}_{l}$.
	 We get therefore $D^{\pi}_{l}$ is an $\dha_m$-module.
	
	We then show that $D^{\pi}_{l}$ is irreducible as an $\dha_m$-module.
	It is enough to show that for any nonzero vector $\nu\in D^{\pi}_{l}$, the $\dha_m$-submodule generated by $\nu$ coincides with $D^{\pi}_{l}$, the former of which is denoted by $N$ in the following.
	
	Firstly, we can write  $$\nu=\underset{\#I=l,J\subset\underline{m}\backslash I}{\sum}\Xi(\omega_{I,J}^{-1})(\nu_{I,J}\otimes\eta_{r+\frac{3}{2}}^{\otimes m-l}),$$
    where $\nu_{I,J}\in S^\pi$. By definition, $\Xi(x_{e}^{(l)})\Xi(\omega_{I,J})(\nu)=\nu_{I,J}\otimes\eta_{r+\frac{3}{2}}^{\otimes m-l}$ for the identity element $e\in\fsl$. Recall that $S^\pi$ is an irreducible
	$\calh(\tsb_{l})$-module,  and
	\begin{align*}
		\Xi(x_{\sigma}^{(l)})(\nu_{I,J}\otimes\eta_{r+\frac{3}{2}}^{\otimes m-l})&=(\Psi_{l}^{V}(H_{\sigma})\nu_{I,J})\otimes\eta_{r+\frac{3}{2}}^{\otimes m-l} \text{ for }\sigma\in \fsl,\cr
		\Xi(T_{0})(\nu_{I,J}\otimes\eta_{r+\frac{3}{2}}^{\otimes m-l})&=(\Psi_{l}^{V}(H_{0})\nu_{I,J})\otimes\eta_{r+\frac{3}{2}}^{\otimes m-l}.
	\end{align*}
	By the action of $\Xi(x_{\sigma}^{(l)})$ with $\sigma$ running over $\fsl$, along with  $\Xi(T_{0})$,
	it is readily known that $N$
	contains the subspace $S^\pi\otimes\eta_{r+\frac{3}{2}}^{\otimes m-l}$, which is the generator of
	 $D^{\pi}_{l}$. Hence $N$ coincides with $D^{\pi}_{l}$. We conclude that
	$D^{\pi}_{l}$ is an irreducible $\dha_m$-module.
\end{proof}

\begin{corollary}\label{cor: 3.4} If $q$ is transcendental
over $\mathbb{Q}$. Then $\underline{V}^{\otimes m}$
is a completely reducible $\dha_m$-module.
\end{corollary}

\begin{proof}
Recall that
$$\underline{V}^{\otimes m}=\bigoplus^m_{l=0}\underline{V}_{l}^{\otimes m}=\bigoplus_{l=0}^{m}\bigoplus_{\pi\in I_{l}} D^{\pi}_{l}.$$
By Lemma \ref{lem: 3.3}, $\underline{V}^{\otimes m}$\textit{ }is a completely
reducible $\dha_m$-module.
\end{proof}

\begin{theorem} \label{thm: main thm2}
If $q$ is transcendental
over $\mathbb{Q}$, then  $\Xi(\dha_m)=\End_{L_{q}^{\imath}(\mathfrak{s}\mathfrak{l}_{2r+4})}(\underline{V}^{\otimes m})^{\rm{op}}$.
\end{theorem}

\begin{proof} By Theorem \ref{thm: main thm1} and Corollary \ref{cor: 3.4}, we
can use classical double commutant theorem (see \cite[\S 4.1.13]{GW})
to reach this conclusion.
\end{proof}

\begin{remark} (1) Recently, Di Wang developed a ``super" version of degenerate duplex Hecke algebras and related Schur-Sergeev duality of Levi-type (see \cite{Wd}).
	
(2) In the case of odd dimensional vector spaces, the duality still holds and the constructions are analogously.
\end{remark}

\end{document}